\documentclass[12pt]{amsart}
\usepackage{amssymb}
\usepackage{amstext}
\usepackage{amsmath}
\usepackage{amscd}
\usepackage{amsthm}
\usepackage{latexsym}
\usepackage{amsfonts}
\usepackage{graphicx}
\usepackage[all]{xy}
\usepackage{enumerate}
\usepackage[usenames]{color}
\newtheorem{thm}{Theorem}[section]
\newtheorem{cor}[thm]{Corollary}
\newtheorem{lem}[thm]{Lemma}
\newtheorem{prop}[thm]{Proposition}
\theoremstyle{definition}
\newtheorem{dfn}[thm]{Definition}

\newtheorem{ex}[thm]{Example}
\newtheorem{nota}[thm]{Notation}
\theoremstyle{remark}
\newtheorem*{ac}{Acknowlegments}
\newtheorem*{conv}{Convention}

\numberwithin{equation}{thm}
\def\A{\mathrm{A}}

\def\cm{\operatorname{\mathsf{CM}}}
\def\cmo{\operatorname{\mathsf{\underline{CM}^o}}}
\def\Cok{\mathrm{Coker}}
\def\D{\mathrm{D}}
\def\db{\operatorname{\mathsf{D^b}}}
\def\ds{\operatorname{\mathsf{D_{sg}}}}
\def\dso{\operatorname{\mathsf{D_{sg}^o}}}

\def\gt{\operatorname{gt}}

\def\ind{\operatorname{\mathsf{ind}}}
\def\lcm{\operatorname{\mathsf{\underline{CM}}}}
\def\lv{\operatorname{level}}
\def\m{\mathfrak{m}}

\def\os{\operatorname{OSpec}}
\def\perf{\operatorname{\mathsf{perf}}}
\def\p{\mathfrak{p}}
\def\s{{R^\sharp}}

\def\T{\mathcal{T}}
\def\syz{\Omega}
\def\Tor{\operatorname{Tor}}
\def\v{\operatorname{V}}
\def\X{\mathcal{X}}
\def\Y{\mathcal{Y}}
\def\Z{\mathbb{Z}}
\begin{document}
\setlength{\baselineskip}{15pt}
\title[Singularity categories of hypersurfaces of countable type]{Generation in singularity categories of hypersurfaces of countable representation type}
\author{Tokuji Araya}
\address[T. Araya]{Department of Applied Science, Faculty of Science, Okayama University of Science, Ridaicho, Kitaku, Okayama 700-0005, Japan}
\email{araya@das.ous.ac.jp}
\author{Kei-ichiro Iima}
\address[K.-i. Iima]{Department of Liberal Studies, National Institute of Technology, Nara College, 22 Yata-cho, Yamatokoriyama, Nara 639-1080, Japan}
\email{iima@libe.nara-k.ac.jp}
\author{Maiko Ono}
\address[M. Ono]{Department of Mathematics, Okayama University, Okayama 700-8530, Japan}
\email{onomaiko@s.okayama-u.ac.jp}
\author{Ryo Takahashi}
\address[R. Takahashi]{Graduate School of Mathematics, Nagoya University, Furocho, Chikusaku, Nagoya 464-8602, Japan/Department of Mathematics, University of Kansas, Lawrence, KS 66045-7523, USA}
\email{takahashi@math.nagoya-u.ac.jp}
\urladdr{https://www.math.nagoya-u.ac.jp/~takahashi/}
\thanks{RT was partly supported by JSPS Grants-in-Aid for Scientific Research 16K05098 and 16KK0099. MO was partly supported by Foundation of Research Fellows, The Mathematical Society of Japan}
\begin{abstract}
The Orlov spectrum and Rouquier dimension are invariants of a triangulated category to measure how big the category is, and they have been studied actively.
In this paper, we investigate the singularity category $\ds(R)$ of a hypersurface $R$ of countable representation type.
For a thick subcategory $\T$ of $\ds(R)$ and a full subcategory $\X$ of $\T$, we calculate the Rouquier dimension of $\T$ with respect to $\X$.
Furthermore, we prove that the level in $\ds(R)$ of the residue field of $R$ with respect to each nonzero object is at most one.
\end{abstract}
\maketitle
\section{Introduction}

The {\em Olrov spectrum} of a triangulated category is introduced by Orlov \cite{O} based on the works of Bondal and Van den Bergh \cite{BV} and Rouquier \cite{R}.
This categorical invariant is the set of finite generation times of objects of the triangulated category.
The generation time of an object is the number of exact triangles necessary to build the category out of the object, up to finite direct sums, direct summands and shifts.
In \cite{R}, Rouquier studies the infimum of the Orlov spectrum, which is called the {\em Rouquier dimension} of the triangulated category.
For more information on Orlov spectra and Rouquier dimensions, we refer the reader to \cite{BFK, O, R} for instance.

The Orlov spectrum and Rouquier dimension measure how big a triangulated category is, but it is basically rather hard to calculate them.
Thus, the notions of a level \cite{ABIM} and a relative Rouquier dimension \cite{AAITY} are introduced to measure how far one given object/subcategory from another given object/subcategory.
The main purpose of this paper is to report on these two invariants for the singularity category of a hypersurface of countable (Cohen--Macaulay) representation type.

Let $k$ be an uncountable algebraically closed field of characteristic not two, and let $R$ be a complete local hypersurface over $k$ with countable representation type.
Denote by $\ds(R)$ the singularity category of $R$, that is, the Verdier quotient of the bounded derived category of finitely generated $R$-modules by the perfect complexes.
Let $\dso(R)$ be the full subcategory of $\ds(R)$ consisting of objects that are zero on the punctured spectrum of $R$.
The main results of this paper are the following two theorems; we should mention that (1) and (2a) of Theorem \ref{b} are essentially shown in the previous papers \cite{AIT,T3} of some of the authors of the present paper.

\begin{thm}\label{a}
For all nonzero objects $M$ of $\ds(R)$ one has
$$
\lv_{\ds(R)}^M(k)\le1,
$$
that is, $k$ belongs to ${\langle M\rangle}_2^{\ds(R)}$.
\end{thm}

\begin{thm}\label{b}
Let $\T$ be a nonzero thick subcategory of $\ds(R)$, and let $\X$ be a full subcategory of $\T$ closed under finite direct sums, direct summands and shifts.
Then the following statements hold.
\begin{enumerate}[\rm(1)]
\item
$\T$ coincides with either $\ds(R)$ or $\dso(R)$.
\item
\begin{enumerate}[\rm(a)]
\item
If $\T=\ds(R)$, then
$$
\dim_\X\T=
\begin{cases}
0 & (\text{if } \X =\T),\\
1 & (\text{if } \X \neq \T,\, \X \nsubseteq \dso(R)),\\
\infty & (\text{if }\X \subseteq \dso(R)).
\end{cases}
$$
\item
If $\T=\dso(R)$, then
$$
\dim_\X\T=
\begin{cases}
0 & (\text{if } \X  =\T),\\
1 & (\text{if } \X \neq \T,\,\#\ind \X=\infty),\\
\infty & (\text{if }\#\ind \X <\infty).
\end{cases}
$$
\end{enumerate}
\end{enumerate}
\end{thm}

As a consequence of Theorem \ref{b}, we get the Orlov spectra and Rouquier dimensions of $\ds(R)$ and $\dso(R)$.
Note that the equalities of Rouquier dimensions are already known to hold \cite{AIT,DT}.

\begin{cor}\label{c}
It holds that
$$
\os\ds(R)=\{1\},\qquad
\os\dso(R)=\emptyset.
$$
In particular, $\dim \ds(R)=1$ and $\dim\dso(R)=\infty$.
\end{cor}

The organization of this paper is as follows.
In Section 2, we give the definitions of the Orlov spectrum and the (relative) Rouquier dimension of a triangulated category.
We also recall the theorem of Buchweitz providing a triangle equivalence for a Gorenstein ring $R$ between the singularity category of $R$ and the stable category of maximal Cohen--Macaulay $R$-modules.
In Section 3, we investigate the Orlov spectra and Rouquier dimensions of the singularity categories of hypersurfaces and their double branched covers to reduce to the case of Krull dimension one.
In Section 4, using the results obtained in the previous sections together with the classification theorem of maximal Cohen--Macaulay modules over a hypersurface of countable representation type, we prove our main theorems stated above.

\begin{conv}
Throughout this paper, all subcategories are assumed to be full.
We often omit subscripts and superscripts if there is no risk of confusion.
\end{conv}
\section{Preliminaries}

We recall the definitions of several basic notions which are used in the later sections.

\begin{nota}
Let $\T$ be a triangulated category.
\begin{enumerate}[(1)]
\item
For a subcategory $\X$ of $\T$ we denote by $\langle \X \rangle$ the smallest subcategory of $\T$ containing $\X$ which is closed under isomorphisms, shifts, finite direct sums and direct summands.
\item
For subcategories $\X, \Y$ of $\T$ we denote by $\X * \Y$ the subcategory consisting of objects $M \in \T$ such that there is an exact triangle $X\to M\to Y \to X[1]$ with $X\in \X$ and $Y\in \Y$.
Set $\X \diamond \Y := \langle\langle\X\rangle * \langle\Y\rangle\rangle$.
\item
For a subcategory $\X$ of $\T$ we put ${\langle \X\rangle}_0:=0$, ${\langle \X\rangle}_1:=\langle \X\rangle$, and inductively define ${\langle \X\rangle}_n:= \X \diamond {\langle \X\rangle}_{n-1}$ for $n\geq 2$.
We set ${\langle M \rangle}_n := {\langle \{ M \}\rangle}_n$ for an object $M\in \T$.
\end{enumerate}
\end{nota}

\begin{dfn}
Let $\T$ be a triangulated category.
\begin{enumerate}[(1)]
\item
The {\em generation time} of an object $M\in\T$ is defined by
$$
\gt_{\T}(M):=\inf\{ n\geq 0\mid\T={\langle M \rangle}_{n+1} \}.
$$
If $\gt_{\T}(M)$ is finite, $M$ is called a {\it strong generator} of $\T$.
\item
The {\it Orlov spectrum} and {\em (Rouquier) dimension} of $\T$ are defined as follows.
\begin{align*}
\os(\T)&:=\{ \gt_{\T}(M)\mid M\text{ is a strong generator of }\T\},\\
\dim\T&:=\inf\os(\T)=\inf\{n\ge0\mid\T={\langle M\rangle}_{n+1}\text{ for some }M\in\T\}.
\end{align*}
\item
Let $\X$ be a subcategory of $\T$.
The {\em dimension} of $\T$ {\em with respect to} $\X$ is defined by
$$
\dim_\X\T:=\inf\{n\ge0\mid \T={\langle\X\rangle}_{n+1}\}.
$$
\item
Let $M, N$ be objects of $\T$.
Then the {\em level} of $N$ with respect to $M$ is defined by
$$
\lv_{\T}^M(N):=\inf\{ n\geq 0\mid N \in {\langle M \rangle}_{n+1} \}.
$$
\end{enumerate}
\end{dfn}

\begin{dfn}
Let $R$ be a Noetherian ring.
\begin{enumerate}[(1)]
\item
We denote by $\db(R)$ the bounded derived category of finitely generated $R$-modules.
\item
A {\em perfect} complex is by definition a bounded complex of finitely generated projective modules.
\item
We denote by $\perf(R)$ the subcategory of $\db(R)$ consisting of complexes quasi-isomorphic to perfect complexes.
\item
The {\em singularity category} of $R$ is defined by
$$
\ds(R):=\db(R)/\perf(R),
$$
that is, the Verdier quotient of $\db(R)$ by $\perf(R)$.
\end{enumerate}
\end{dfn}

Note that every object of the singularity category $\ds(R)$ is isomorphic to a shift of some $R$-module; see \cite[Lemma 2.4]{sing}.

Let $R$ be a Cohen--Macaulay local ring.
Let $\cm(R)$ be the category of maximal Cohen-Macaulay $R$-modules, and $\lcm(R)$ the stable category of $\cm(R)$.
The following theorem is celebrated and fundamental; see \cite[Theorem 4.4.1]{B}.

\begin{thm}[Buchweitz]\label{ragnar}
Let $R$ be a Gorenstein local ring of Krull dimension $d$.
Then $\lcm(R)$ has the structure of a triangulated category with shift functor $\Omega^{-1}$, and there exist mutually inverse triangle equivalence functors
$$
F:\ds(R)\rightleftarrows\lcm(R):G,
$$
such that $GM=M$ for each maximal Cohen--Macaulay $R$-module $M$ and $FN=\syz^dN[d]$ for each finitely generated $R$-module $N$.
\end{thm}

By virtue of Theorem \ref{ragnar}, for a Gorenstein local ring, the study of generation in the singularity category reduces to the stable category of maximal Cohen--Macaulay modules.

\section{The relationship between the singularity categories of $R$ and $\s$}

Let $(R,\m,k)$ be a complete equicharacteristic local hypersurface of (Krull) dimension $d$.
Then thanks to Cohen's structure theorem we can identify $R$ with a quotient of a formal power series ring over $k$:
$$
R=k[\![x_0,x_1,\dots,x_d]\!]/(f)
$$
with $0\neq f \in (x_0,x_1,\dots,x_d)^2$. 
We define a hypersurface of dimension $d+1$:
$$
\s=k[\![x_0,x_1,\ldots,x_d, y]\!]/(f+y^2).
$$
Note that the element $y$ is $\s$-regular and there is an isomorphism $\s/y\s\cong R$.
The main purpose of this section is to compare generation in the singularity categories $\ds(R)$ and $\ds(\s)$.
As both $R$ and $\s$ are Gorenstein, in view of Theorem \ref{ragnar} and the remark following the theorem, it suffices to investigate the stable categories of maximal Cohen--Macaulay modules $\lcm(R)$ and $\lcm(\s)$.

The following result is a consequence of \cite[Proposition 12.4]{Y}, which plays a key role to compare generation in $\lcm(R)$ and $\lcm(\s)$.

\begin{lem}\label{dense}
The assignments $M\mapsto\syz_{\s}M$ and $N\mapsto N/yN$ define triangle functors $\Phi :\lcm(R) \to \lcm(\s)$ and $\Psi :\lcm(\s) \to \lcm(R)$ satisfying
$$
\Psi\Phi(M)\cong M\oplus M[1],\qquad
\Phi\Psi(N) \cong N\oplus N[1].
$$
In particular, $\Phi$ and $\Psi$ are both equivalences up to direct summands.
\end{lem}

Applying this lemma, we deduce relationships of levels in $\lcm(R)$ and $\lcm(\s)$.

\begin{prop}\label{lev}
One has the following equalities.
\begin{enumerate}[\rm(1)]
\item
$\lv_{\lcm(R)}^M(\syz_R^dk)=\lv_{\lcm(\s)}^{\syz_{\s} M}(\syz_{\s}^{d+1}k)$ for each $M\in\lcm(R)$.
\item
$\lv_{\lcm(\s)}^N(\syz_{\s}^{d+1}k)=\lv_{\lcm(R)}^{N/yN}(\syz_{R}^{d}k)$ for each $N\in\lcm(\s)$.
\end{enumerate}
\end{prop}

\begin{proof}
We use Lemma \ref{dense} and adopt its notation.
There are (in)equalities
\begin{align*}
\lv_{\lcm(R)}^M(\syz_R^dk)
\ge\lv_{\lcm(\s)}^{\Phi M}(\Phi(\syz_R^dk))
&\ge\lv_{\lcm(R)}^{\Psi\Phi M}(\Psi\Phi(\syz_R^dk))\\
&=\lv_{\lcm(R)}^{M\oplus M[1]}(\syz_R^dk\oplus\syz_R^dk[1])
=\lv_{\lcm(R)}^M(\syz_R^dk),
\end{align*}
which show $\lv_{\lcm(R)}^M(\syz_R^dk)=\lv_{\lcm(\s)}^{\Phi M}(\Phi(\syz_R^dk))$.
A similar argument gives rise to $\lv_{\lcm(\s)}^N(\syz_{\s}^{d+1}k)=\lv_{\lcm(R)}^{\Psi N}(\Psi(\syz_\s^{d+1}k))$.
There are isomorphisms in $\lcm(R)$:
\begin{align}\label{1e}
\Psi(\syz_\s^{d+1}k)
=\syz_\s^{d+1}k/y\syz_\s^{d+1}k
&\cong\syz_{\s/y\s}^{d+1}k\oplus\syz_{\s/y\s}^dk\\
\notag&\cong\syz_R^{d+1}k\oplus\syz_R^dk
\cong\syz_R^dk[1]\oplus\syz_R^dk,
\end{align}
where the first isomorphism follows from \cite[Corollary 5.3]{T0}.
Applying $\Phi$, we obtain
\begin{equation}\label{2e}
\syz_\s^{d+1}k\oplus\syz_\s^{d+1}k[1]\cong\Phi(\syz_R^dk)[1]\oplus\Phi(\syz_R^dk).
\end{equation}
It is observed from \eqref{1e} and \eqref{2e} that $\lv_{\lcm(R)}^{\Psi N}(\Psi(\syz_\s^{d+1}k))=\lv_{\lcm(R)}^{\Psi N}(\syz_R^dk)$ and $\lv_{\lcm(\s)}^{\Phi M}(\Phi(\syz_R^dk))=\lv_{\lcm(\s)}^{\Phi M}(\syz_\s^{d+1}k)$, respectively.
Consequently we obtain
\begin{align*}
\lv_{\lcm(R)}^M(\syz_R^dk)
&=\lv_{\lcm(\s)}^{\Phi M}(\Phi(\syz_R^dk))
=\lv_{\lcm(\s)}^{\Phi M}(\syz_\s^{d+1}k)
=\lv_{\lcm(\s)}^{\syz_\s M}(\syz_\s^{d+1}k),\\
\lv_{\lcm(\s)}^N(\syz_{\s}^{d+1}k)
&=\lv_{\lcm(R)}^{\Psi N}(\Psi(\syz_\s^{d+1}k))
=\lv_{\lcm(R)}^{\Psi N}(\syz_R^dk)
=\lv_{\lcm(R)}^{N/yN}(\syz_R^dk),
\end{align*}
which completes the proof of the proposition.
\end{proof}

Using Lemma \ref{dense} again, we get relationships of generation times in $\lcm(R)$ and $\lcm(\s)$.

\begin{prop}\label{prop1}
The following statements hold true.
\begin{enumerate}[\rm(1)]
\item
If $M \in \lcm(R)$ is a strong generator, then so is $\syz_\s M\in\lcm(\s)$, and $\gt_{\lcm(R)}(M) = \gt_{\lcm(\s)}(\syz_\s M)$.
\item
If $N \in \lcm(\s)$ is a strong generator, then so is $N/yN\in\lcm(R)$, and $\gt_{\lcm(\s)}(N) = \gt_{\lcm(R)}(N/yN)$.
\end{enumerate}
\end{prop}

\begin{proof}
(1) We use Lemma \ref{dense} and adopt its notation.
Put $n=\gt_{\lcm(R)}(M)$.
By definition, it holds that ${\langle M\rangle}_{n+1}=\lcm(R)\ne{\langle M\rangle}_n$.
What we need to prove is that ${\langle\Phi M\rangle}_{n+1}=\lcm(\s)\ne{\langle\Phi M\rangle}_n$.
For each $X\in\lcm(\s)$ we have $\Psi X\in\lcm(R)={\langle M\rangle}_{n+1}$, and $\Phi\Psi X\in{\langle\Phi M\rangle}_{n+1}$.
Since $X$ is a direct summand of $\Phi\Psi X$, it belongs to ${\langle\Phi M\rangle}_{n+1}$.
Therefore, we get $\lcm(\s)={\langle\Phi M\rangle}_{n+1}$.
Suppose that the equality $\lcm(\s)={\langle\Phi M\rangle}_n$ holds.
Taking any $Y\in\lcm(R)$, we have $\Phi Y\in\lcm(\s)={\langle\Phi M\rangle}_n$, and $\Psi\Phi Y\in{\langle\Psi\Phi M\rangle}_n$.
As $Y$ is a direct summand of $\Psi\Phi Y$, it is in ${\langle\Psi\Phi M\rangle}_n$.
Hence $\lcm(R)={\langle\Psi\Phi M\rangle}_n={\langle M\oplus M[1]\rangle}_n={\langle M\rangle}_n$, which is a contradiction.
Thus $\lcm(\s)\ne{\langle\Phi M\rangle}_n$.

(2) An analogous argument to the proof of (1) applies.
\end{proof}

The Orlov spectra and Rouquier dimensions of $\lcm(R)$ and $\lcm(\s)$ coincide:

\begin{cor}\label{cor}
One has the following equalities.
$$
\os\lcm(R)=\os\lcm(\s),\qquad
\dim\lcm(R)=\dim\lcm(\s).
$$
\end{cor}

\begin{proof}
It suffices to show the first equality, since the second equality follows by taking the infimums of the both sides of the first equality.
Using Proposition \ref{prop1}(1), we obtain
\begin{align*}
\os\lcm(R)
&=\{\gt_{\lcm(R)}(M)\mid\text{$M$ is a strong generator of $\lcm(R)$}\}\\
&=\{\gt_{\lcm(\s)}(\syz_{\s}M)\mid\text{$\syz_{\s}M$ is a strong generator of $\lcm(\s)$}\}\\
&\subseteq\{\gt_{\lcm(\s)}(N)\mid\text{$N$ is a strong generator of $\lcm(\s)$}\}
=\os\lcm(\s).
\end{align*}
A similar argument using Proposition \ref{prop1}(2) shows the opposite inclusion $\os\lcm(\s)\subseteq\os\lcm(R)$.
We thus conclude that $\os\lcm(R)=\os\lcm(\s)$.
\end{proof}

\section{The singularity category of a hypersurface of countable representation type}

In this section, we prove our main results, that is, Theorems \ref{B} and \ref{A} from the Introduction.
We start by the following lemma on exact triangles in a triangulated category (an exact triangle $A\to B\to C\to A[1]$ is simply denoted by $A\to B\to C\rightsquigarrow$).

\begin{lem}\label{ladder}
Let $\T$ be a triangulated category.
Let
$$
X\xrightarrow{{\left(\begin{smallmatrix}f_1\\f_2\end{smallmatrix}\right)}}Y_1\oplus M\xrightarrow{{\left(\begin{smallmatrix}g_1&\alpha\end{smallmatrix}\right)}}N\overset{p}\rightsquigarrow,\qquad
M\xrightarrow{{\left(\begin{smallmatrix}\alpha\\g_2\end{smallmatrix}\right)}}N\oplus Y_2\xrightarrow{\left(\begin{smallmatrix}h_1&h_2\end{smallmatrix}\right)}Z\overset{q}\rightsquigarrow
$$
be exact triangles in $\T$.
Then there exists an exact triangle in $\T$ of the form
$$
X\xrightarrow{\left(\begin{smallmatrix}f_1\\g_2f_2\end{smallmatrix}\right)}Y_1\oplus Y_2\xrightarrow{\left(\begin{smallmatrix}h_1g_1&-h_2\end{smallmatrix}\right)}Z\rightsquigarrow.
$$
\end{lem}

\begin{proof}
There is an isomorphism
$$
\xymatrix@R=3pc@C=5pc{
M\ar[r]^-{\left(\begin{smallmatrix}
0\\
1\\
0
\end{smallmatrix}\right)}\ar@{=}[d] & Y_1\oplus M\oplus Y_2\ar[r]^-{\left(\begin{smallmatrix}
1 & 0 & 0\\
0 & 0 & 1
\end{smallmatrix}\right)}\ar[d]_{\cong}^-{\left(\begin{smallmatrix}1&0&0\\0&1&0\\0&g_2&1\end{smallmatrix}\right)} & Y_1\oplus Y_2\ar[r]^-{\left(\begin{smallmatrix}
0 & 0
\end{smallmatrix}\right)}\ar[d]_{\cong}^-{\left(\begin{smallmatrix}1&0\\0&-1\end{smallmatrix}\right)} & M[1]\ar@{=}[d] \\
M\ar[r]^-{\left(\begin{smallmatrix}
0\\
1\\
g_2
\end{smallmatrix}\right)} & Y_1\oplus M\oplus Y_2\ar[r]^-{\left(\begin{smallmatrix}
1 & 0 & 0\\
0 & g_2 & -1
\end{smallmatrix}\right)} & Y_1\oplus Y_2\ar[r]^-{\left(\begin{smallmatrix}0&0\end{smallmatrix}\right)} & M[1] 
}
$$
of sequences.
The first row is an exact triangle in $\T$ since it is the direct sum of exact triangles arising from the identity maps of $M$ and $Y_1\oplus Y_2$ (see \cite[Proof of Corollary 1.2.7]{N}).
Hence the second row is an exact triangle in $\T$ as well.
We have a commutative diagram
$$
\xymatrix@R=3pc@C=5pc{
& X\ar@{=}[r]\ar[d]^-{\left(\begin{smallmatrix}
f_1\\
f_2\\
0
\end{smallmatrix}\right)} & X\ar@{..>}[d]^-{\left(\begin{smallmatrix}
s\\
t
\end{smallmatrix}\right)}\\
M\ar[r]^-{\left(\begin{smallmatrix}
0\\
1\\
g_2
\end{smallmatrix}\right)}\ar@{=}[d] & Y_1\oplus M \oplus Y_2\ar[r]^-{\left(\begin{smallmatrix}
1 & 0 & 0\\
0 & g_2 & -1
\end{smallmatrix}\right)}\ar[d]^-{\left(\begin{smallmatrix}g_1&\alpha&0\\0&0&1\end{smallmatrix}\right)} & Y_1\oplus Y_2\ar[r]^{\left(\begin{smallmatrix}0&0\end{smallmatrix}\right)}\ar@{..>}[d]^-{\left(\begin{smallmatrix}u&v\end{smallmatrix}\right)} & M[1]\ar@{=}[d]\\
M\ar[r]^-{\left(\begin{smallmatrix}
\alpha\\
g_2
\end{smallmatrix}\right)} & N \oplus Y_2\ar[r]^-{\left(\begin{smallmatrix}
h_1 & h_2 
\end{smallmatrix}\right)}\ar[d]^-{\left(\begin{smallmatrix}p&0\end{smallmatrix}\right)} & Z\ar[r]^-q\ar@{..>}[d] & M[1]\\
& X[1]\ar@{=}[r] & X[1]
}
$$
where the rows are exact triangles in $\T$, and so is the left column since it is a direct sum of exact triangles (see \cite[Proposition 1.2.1]{N}).
Using the octahedral axiom, we obtain the right column which is an exact triangle in $\T$.
The diagram chasing shows that $\left(\begin{smallmatrix}s\\t\end{smallmatrix}\right)=\left(\begin{smallmatrix}f_1\\g_2f_2\end{smallmatrix}\right)$ and $\left(\begin{smallmatrix}u&v\end{smallmatrix}\right)=\left(\begin{smallmatrix}h_1g_1&-h_2\end{smallmatrix}\right)$.
Thus it is an exact triangle we want.
\end{proof}

Let $(R,\m,k)$ be a complete equicharacteristic local hypersurface of dimension $d$.
Assume that $k$ is uncountable and has characteristic different from two, and that $R$ has {\em countable (Cohen--Macaulay) representation type}, namely, there exist infinitely but only countably many isomorphism classes of indecomposable maximal Cohen--Macaulay $R$-modules. 
Then $f$ is either of the following; see \cite[Theorem 14.16]{LW}.
\begin{align*}
(\A_\infty) &:\ x_0^2+x_2^2+\cdots +x_d^2,\\
(\D_\infty) &:\ x_0^2x_1+x_2^2+\cdots +x_d^2.
\end{align*}
In this case, all objects in $\cm(R)$ are completely classified \cite{BGS,BD,K}.

Now we can state and prove the following result regarding levels in $\lcm(R)$.

\begin{thm}\label{B}
Let $k$ be an uncountable algebraically closed field of characteristic not two.
Let $R$ be a $d$-dimensional complete local hypersurface over $k$ of countable representation type.
Then
$$
\syz_R^dk\in{\langle M\rangle}_2^{\lcm(R)}
$$
for all nonzero objects $M\in\lcm(R)$.
In other words, $\lv_{\lcm(R)}^M(\syz_R^dk)\le1$.
\end{thm}

\begin{proof}
Proposition \ref{lev}(2) reduces to the case $d=1$.
Thus we have the two cases:
$$
\text{(1) }R=k[\![x,y]\!]/(x^2),\qquad
\text{(2) }R=k[\![x,y]\!]/(x^2y).
$$

(1):
Thanks to \cite[4.1]{BGS}, the indecomposable objects of $\lcm(R)$ are the ideals $I_n=(x,y^n)$ with $n\in\Z_{>0}\cup\{\infty\}$, where $I_{\infty}:=(x)$.
By \cite[6.1]{Sc} there exist exact triangles
$$
I_{n}\to I_{n-1}\oplus I_{n+1}\to I_{n}\rightsquigarrow\quad(n\in\Z_{>0}),
$$
where $I_0:=0$.
Applying Lemma \ref{ladder}, we obtain exact triangles
$$
I_{n}\to I_{1}\oplus I_{2n-1}\to I_{n}\rightsquigarrow\quad(n\in\Z_{>0}),
$$
and from \cite[Proposition 2.1]{AIT} we obtain an exact triangle $I_{\infty}\to I_{1}\to I_{\infty}\rightsquigarrow$.
It is observed from these triangles that $\syz k=I_1$ is in ${\langle M\rangle}_2$ for each nonzero object $M\in\lcm(R)$.

(2):
Using \cite[4.2]{BGS}, we get a complete list of the indecomposable objects of $\lcm(R)$:
\begin{align*}
&X=R/(x),
\ \syz_RX,
\ M_0^+=R/(x^2),
\ M_0^-=\syz M_0^+,\\
&M_n^+=\Cok\left(\begin{smallmatrix}
x&y^n\\
0&-x
\end{smallmatrix}\right),
\ M_n^-=\syz M_n^+,
\ N_n^+=\Cok\left(\begin{smallmatrix}
x&y^n\\
0&-xy
\end{smallmatrix}\right),
\ N_n^{-}=\syz N_n^{+} \quad(n\in\Z_{>0}).
\end{align*}
According to \cite[(6.1)]{Sc}, for each $n\in\Z_{>0}$ there are exact triangles
$$
M_n^{\pm}\to N_{n+1}^{\pm}\oplus N_{n}^{\mp}\to M_n^{\mp}\rightsquigarrow,\qquad
N_n^{\pm}\to M_{n}^{\pm}\oplus M_{n-1}^{\mp}\to N_n^{\mp}\rightsquigarrow,
$$
where $N_0^{\pm}:=0$.
Lemma \ref{ladder} gives rise to exact triangles
$$
M_n\to N_1 \oplus N_{2n} \to M_n\rightsquigarrow,\qquad
N_n\to N_1 \oplus N_{2n-1} \to N_n\rightsquigarrow,
$$
where $M_n$ stands for either $M_n^+$ or $M_n^-$, and so on.
Also, by \cite[Proposition 2.1]{AIT} we get an exact triangle $\syz X\to N_{1}^-\to X\rightsquigarrow$.
Thus $\syz k=N_1^-$ is in ${\langle C\rangle}_2$ for any $0\ne C\in\lcm(R)$.
\end{proof}

\begin{proof}[\bf Proof of Theorem \ref{a}]
The assertion is immediate from Theorems \ref{B} and \ref{ragnar}.
\end{proof}

The following example shows that Theorem \ref{a} does not necessarily hold if one replaces $k$ with another nonzero object of the singularity category.

\begin{ex}\label{3}
Let $R=k[\![x,y]\!]/(x^2)$ be a hypersurface over a field $k$.
Then
$$
(x,y^a)R\notin{\left\langle(x,y^b)R\right\rangle}_2^{\ds(R)}
$$
for all positive integers $a,b$ with $a>2b$.
Thus $\lv_{\ds(R)}^{(x,y^b)R}((x,y^a)R)\ge2$.
\end{ex}

\begin{proof}
In view of Theorem \ref{ragnar} we replace $\ds(R)$ with $\lcm(R)$.
Let $I=(x,y^a)R$ and $J=(x,y^b)R$ be ideals of $R$.
Suppose that $I$ belongs to ${\langle J\rangle}_2^{\lcm(R)}$.
Since $\syz J\cong J$, we see that there exists an exact sequence $0 \to J^{\oplus m} \to I\oplus M \to J^{\oplus n} \to 0$ of $R$-modules with $m,n\ge0$.
This induces an exact sequence
$$
\Tor_1^R(R/I,J^{\oplus m}) \to \Tor_1^R(R/I,I)\oplus\Tor_1^R(R/I,M) \to \Tor_1^R(R/I,J^{\oplus n}).
$$
Since $\Tor_1^R(R/I,J)\cong\Tor_2^R(R/I,R/J)$, the first and third Tor modules in the above exact sequence are annihilated by $J$, and hence
\begin{equation}\label{1}
J^2\Tor_1^R(R/I,I)=0.
\end{equation}
The minimal free resolution of $R/I$ is
$$
F=(\cdots\xrightarrow{\left(\begin{smallmatrix}
y^a&x\\
-x&0
\end{smallmatrix}\right)}R^{\oplus2}\xrightarrow{\left(\begin{smallmatrix}
0&-x\\
x&y^a
\end{smallmatrix}\right)}R^{\oplus2}\xrightarrow{\left(\begin{smallmatrix}
y^a&x\\
-x&0
\end{smallmatrix}\right)}R^{\oplus2}\xrightarrow{(x,y^a)}R\to0),
$$
which induces a complex
$$
F\otimes_RR/I=(\cdots\xrightarrow{0}(R/I)^{\oplus2}\xrightarrow{0}(R/I)^{\oplus2}\xrightarrow{0}(R/I)^{\oplus2}\xrightarrow{0}R/I\to0).
$$
Hence $\Tor_1^R(R/I,I)\cong\Tor_2^R(R/I,R/I)=(R/I)^{\oplus2}$.
By \eqref{1} we have $J^2(R/I)^{\oplus2}=0$, which implies that $J^2$ is contained in $I$.
Therefore the element $y^{2b}$ is in the ideal $(x,y^a)R$, but this cannot happen since $a>2b$.
\end{proof}

Recall that a subcategory of a triangulated category is called {\em thick} if it is a triangulated subcategory closed under direct summands.
We denote by $\cmo(R)$ the subcategory of $\lcm(R)$ consisting of maximal Cohen--Macaulay $R$-modules that are free on the punctured spectrum of $R$.
The category $\cmo(R)$ is a thick subcategory of $\lcm(R)$, and in particular it is a triangulated category.
For a subcategory $\X$ of $\lcm(R)$ we denote by $\ind\X$ the set of nonisomorphic indecomposable objects of $\lcm(R)$ that belong to $\X$.
We can now state and prove the following result concerning relative Rouquier dimensions in $\lcm(R)$.

\begin{thm}\label{A}
Let $k$ be an uncountable algebraically closed field of characteristic not two.
Let $R$ be a $d$-dimensional complete local hypersurface over $k$ of countable representation type.
Let $\T\ne0$ be a thick subcategory of $\lcm(R)$, and let $\X$ be a subcategory of $\T$ closed under finite direct sums, direct summands and shifts.
Then:
\begin{enumerate}[\rm(1)]
\item
$\T$ coincides with either $\lcm(R)$ or $\cmo(R)$.
\item
\begin{enumerate}[\rm(a)]
\item
When $\T=\lcm(R)$, one has
$$
\dim_\X\T=
\begin{cases}
0 & (\text{if } \X=\T),\\
1 & (\text{if } \X \neq \T\text{ and } \X \nsubseteq \cmo(R)),\\
\infty & (\text{if } \X \subseteq \cmo(R)).
\end{cases}
$$
\item
When $\T=\cmo(R)$, one has
$$
\dim_\X\T=
\begin{cases}
0 & (\text{if } \X=\T),\\
1 & (\text{if } \X\neq \T\text{ and }\#\ind \X =\infty),\\
\infty & (\text{if }\#\ind \X<\infty).
\end{cases}
$$
\end{enumerate}
\end{enumerate}
\end{thm}

\begin{proof}
(1) We combine \cite[Theorem 6.8]{T3} and \cite[Theorem 1.1]{AIT}.
The singular locus of $R$ consists of two points $\p$ and $\m$, and its specialization-closed subsets are $\v(\p)$, $\v(\m)$ and $\emptyset$.
These correspond to the thick subcategories $\lcm(R)$, $\cmo(R)$ and $0$.

(2) Part (a) follows from \cite[Theorem 1.1]{AIT}.
Let us show part (b).
When $\#\ind \X<\infty$, let $X_1,\dots,X_n$ be all the indecomposable objects in $\X$.
Suppose that $\dim_\X\cmo(R)$ is finite, say $m$.
Then it follows that $\cmo(R)={\langle\X\rangle}_{m+1}={\langle X\rangle}_{m+1}$, where $X:=X_1\oplus\cdots\oplus X_n\in\cmo(R)$.
Hence $\cmo(R)$ has finite Rouquier dimension.
By \cite[Theorem 1.1(2)]{DT}, the local ring $R$ has to have at most an isolated singularity.
However, in either case of the types $(\A_\infty)$ and $(\D_\infty)$ we see that the nonmaximal prime ideal $(x_0,x_2,\dots,x_d)R$ belongs to the singular locus of $R$, which is a contradiction.
Consequently, we obtain $\dim_\X\cmo(R)=\infty$.

From now on we consider the case where $\X\ne\T=\cmo(R)$ and $\#\ind\X=\infty$.
We adopt the same notation as in the proof of Theorem \ref{B}.

Assume that $R$ has type $(A_\infty)$.
As $\X$ is a proper subcategory, we can find a positive integer $n$ such that $I_n\notin\X$.
Since there are infinitely many indecomposable objects in $\X$, we can also find an integer $m>n$ such that $I_m\in\X$.
There exists an exact triangle
$$
I_{m}\to I_{n}\oplus I_{2m-n}\to I_{m}\rightsquigarrow
$$
in $\T=\cmo(R)$, which shows that $I_n$ belongs to ${\langle \X\rangle}_2$.
Therefore, we get $\dim_\X\T\le1$.
Since $\langle\X\rangle=\X\neq \T$, we have $\dim_\X\T\ne 0$.
Consequently, we obtain $\dim_\X\T=1$.

Suppose that $R$ is of type $(D_\infty)$.
Similarly as above, we find two integers $m>n>0$ such that neither $M_n$ nor $N_n$ belongs to $\X$ and either $M_m$ or $N_m$ is in $\X$.
When $M_m$ belongs to $\X$, there are exact triangles
$$
M_{m}\to M_{n}\oplus M_{2m-n}\to M_{m}\rightsquigarrow,\qquad
M_{m}\to N_{n}\oplus N_{2m-n+1}\to M_{m}\rightsquigarrow.
$$
When $N_m$ is in $\X$, we have exact triangles
$$
N_{m}\to M_{n}\oplus M_{2m-n-1}\to N_{m}\rightsquigarrow,\qquad
N_{m}\to N_{n}\oplus N_{2m-n}\to N_{m}\rightsquigarrow.
$$
In either case, both $M_n$ and $N_n$ belong to ${\langle \X\rangle}_2$.
It follows that $\dim_\X\T\le1$.
As $\langle\X\rangle=\X\neq \T$, we have $\dim_\X\T\ne 0$.
Now we conclude $\dim_\X\T=1$.
\end{proof}

\begin{proof}[\bf Proof of Theorem \ref{b}]
Theorems \ref{A} and \ref{ragnar} immediately deduce the assertion.
\end{proof}

As a corollary of Theorem \ref{A}, we calculate the Orlov spectra and Rouquier dimensions of $\lcm(R)$ and $\cmo(R)$ for a hypersurface $R$ of countable representation type.

\begin{cor}\label{x}
Let $k$ be an algebraically closed uncountable field of characteristic not two. 
Let $R$ be a $d$-dimensional complete local hypersurface over $k$ having countable representation type.
Then one has the following equalities.
\begin{enumerate}[\rm(1)]
\item
$\os\lcm(R)=\{1\}$.
\item
$\os\cmo(R)=\emptyset$.
\end{enumerate}
In particular, $\dim\lcm(R)=1$ and $\dim\cmo(R)=\infty$.
\end{cor}

\begin{proof}
(1) Applying Theorem \ref{A}(2a) for $\X=\langle M\rangle$, we see that it suffices to verify that
\begin{enumerate}[(a)]
\item
$\langle M\rangle\ne\lcm(R)$ for all $M\in\lcm(R)$, and
\item
$\langle M\rangle\ne\lcm(R)$ for some $M\in\lcm(R)\setminus\cmo(R)$.
\end{enumerate}
Statement (a) is equivalent to saying that $\dim\lcm(R)>0$, which follows from \cite[Propositions 2.4 and 2.5]{radius}.
There exists an object $M\in\lcm(R)\setminus\cmo(R)$ since $R$ does not have an isolated singularity, and hence (b) follows from (a).

(2) Let $M\in\cmo(R)$ be an object, and let $\X=\langle M\rangle$ be a subcategory of $\cmo(R)$.
Then it is observed that $\#\ind\X<\infty$, so applying Theorem \ref{A}(2b), we obtain $\dim_\X\cmo(R)=\infty$, which means $\gt_{\cmo(R)}(M)=\infty$.
Thus the assertion follows.
\end{proof}

\begin{proof}[\bf Proof of Corollary \ref{c}]
Combining Corollary \ref{x} with Theorem \ref{ragnar} yields the assertion.
\end{proof}

\begin{ac}
The authors thank the referee for reading the paper carefully and giving useful suggestions.
\end{ac}

\end{document}